\documentclass{amsart}%kh
\usepackage{amssymb}
\usepackage{amsmath}
\usepackage{latexsym}
\usepackage{mathtools}
\usepackage{eepic}
\usepackage{epsfig}
\usepackage{graphicx}
\usepackage{mathrsfs}
\usepackage[english]{babel}
\usepackage{todonotes}
\usepackage{mathtools}

\usetikzlibrary{shapes}
\DeclareMathAlphabet{\mathpzc}{OT1}{pzc}{m}{it}
\newtheorem{theorem}[equation]{Theorem}
\newtheorem{theorem-definition}[equation]{Theorem-Definition}
\newtheorem{lemma-definition}[equation]{Lemma-Definition}
\newtheorem{definition-prop}[equation]{Proposition-Definition}

\newtheorem{prop}[equation]{Proposition}
\newtheorem{lemma}[equation]{Lemma}
\newtheorem{cor}[equation]{Corollary}

\theoremstyle{definition}
\newtheorem{definition}[equation]{Definition}

\newtheorem{example}[equation]{Example}

\newcommand{\N}{\ensuremath{\mathbb{N}}}
\newcommand{\Z}{\ensuremath{\mathbb{Z}}}

\newcommand{\R}{\ensuremath{\mathbb{R}}}
\newcommand{\C}{\ensuremath{\mathbb{C}}}

\renewcommand{\O}{\mathcal{O}}

\newcommand{\A}{\ensuremath{\mathbb{A}}}

\newcommand{\cX}{\ensuremath{\mathcal{X}}}

\newcommand{\cZ}{\ensuremath{\mathcal{Z}}}

\newcommand{\cY}{\ensuremath{\mathcal{Y}}}

\newcommand{\red}{\ensuremath{\mathrm{red}}}

\DeclareMathOperator{\Spec}{Spec}
\DeclareMathOperator{\lcm}{lcm}

\DeclareMathOperator{\Tr}{Tr}

\newcommand{\spe}{\ensuremath{\mathrm{sp}}}

\numberwithin{equation}{section} \hyphenpenalty=6000
\newcommand{\sss}{\vspace{5pt} \subsection*{ }\refstepcounter{equation}{{\bfseries(\theequation)}\ }}

\title{The specialization index of a variety over a discretely valued field}

\author[Lore Kesteloot]{Lore Kesteloot}
\address{KULeuven\\
Department of Mathematics\\ Celestijnenlaan 200B\\3001 Heverlee \\
Belgium} \email{lore.kesteloot@wis.kuleuven.be}

\author[Johannes Nicaise]{Johannes Nicaise}
\address{KULeuven\\
Department of Mathematics\\ Celestijnenlaan 200B\\3001 Heverlee \\
Belgium} \email{johannes.nicaise@wis.kuleuven.be}

\begin{document}
\begin{abstract}
Let $X$ be a proper variety over a henselian discretely valued field. An important obstruction to the existence of a rational point on $X$ is the {\em index}, the minimal positive degree of a zero-cycle on $X$.
 This paper introduces a new invariant, the
{\em specialization index}, which is a closer approximation of the existence of a rational point.
 We provide
an explicit formula for the specialization index in terms of an $snc$-model,
 and we give examples of curves where the index equals one but the specialization index is different from one, and thus explains the absence of a rational point.
 Our main result states that the specialization index of a smooth, proper, geometrically connected $\C((t))$-variety with trivial coherent cohomology is equal to one.
\end{abstract}
\maketitle

\section{Introduction}\label{sec-intro}
Let $R$ be a henselian discrete valuation ring with quotient field
$K$ and algebraically closed residue field $k$. The $C_1$ conjecture predicts that every separably rationally connected $K$-variety has a rational point\footnote{To be precise, the field $K$ is not necessarily $C_1$; but on smooth $K$-varieties, the existence of a rational point is equivalent to the existence of a point over the completion of $K$, which is a $C_1$ field.}.
 This is known when $R$ has equal characteristic: by Greenberg approximation one can assume that the variety is defined over a function field, in which case the statement was proven by Grabber-Harris-Starr in characteristic zero \cite{GHS} and de Jong-Starr in positive characteristic \cite{jongstarr2003} by means of geometric deformation arguments. However, no direct proof (avoiding Greenberg approximation) is known, and  the mixed characteristic case is still wide open. For these purposes, it is desirable to obtain a motivic understanding of the result in equal characteristic, which would have better chances to generalize to mixed characteristic. A fundamental result in this direction is Esnault's proof of the existence of a rational point on rationally chain connected varieties over finite fields \cite{esnault2003}.

 An important obstruction to the existence of a rational point on a $K$-scheme of finite type $X$ is given by the
  \textit{index}
$\iota(X)$. It is the greatest common
divisor of the degrees of the closed points on $X$, or, equivalently, the minimum of the set of integers $d>0$ such that $X$ has a zero-cycle
of degree $d$. If $X$ has a $K$-rational point, then $\iota(X)$ equals one, but the converse does not hold in general.
 In many interesting situations, the index can be computed in a motivic or cohomological way. For instance, Esnault, Levine and Wittenberg
 have proven that, if $X$ is proper over $K$ and we denote by $p\geq 1$ the characteristic exponent of $k$, then the prime-to-$p$ part of the index $\iota(X)$ divides the Euler characteristic of any coherent sheaf on $X$ \cite[3.1]{esnlevwit}. It follows that, when $k$ has characteristic zero, every rationally connected $K$-variety $X$ has index one, since the higher coherent cohomology spaces of $X$ vanish so that the coherent Euler characteristic of $X$ equals one. The latter result has been proven independently by several authors; we refer to the introduction of \cite{esnlevwit} for precise references. Beware, however, that the triviality of the coherent cohomology is not sufficient to guarantee the existence of a rational point: in \cite{Lafon}, Lafon has constructed an example of an Enriques surface over $\C((t))$ without rational points.

 The purpose of the present paper is to introduce a finer invariant than the index, which we call the {\em specialization index}, for proper $K$-schemes of finite type. It is always bounded by the index and it equals one if $X$ has a rational point. The interest of this invariant lies in the fact that there exist many proper $K$-schemes of index one whose specialization index is strictly larger than one. In such examples, the specialization index explains the non-existence of a rational point, whereas the index does not. Our main result (Theorem \ref{thm-zero}) states that, when $k$ has characteristic zero and $X$ is a smooth, proper, geometrically connected $K$-variety with trivial coherent cohomology, then the specialization index of $X$ equals one. This strengthens the aforementioned result for the index. In a sense, this result is disappointing, because it means that the specialization index does not allow to distinguish between varieties with and without rational points in the case of trivial coherent cohomology.
 Our proof is inspired by, and makes use of, the Woods Hole trace formula \cite[III.6.12]{sga5}, which implies that over a {\em finite} field, every smooth and proper variety with trivial coherent cohomology has a rational point.

 Let us give a brief overview of the structure of the paper.
 In Section \ref{sec-def}, we will define the specialization index and give
an explicit way to compute it in terms of the special fiber of an $snc$-model.
 We then use this formula, together with Winters's classification of combinatorial reduction types of curves, to produce examples
of curves $C$ with $\iota(C)=1$ and $\iota_\spe(C)>1$ in Section \ref{sec-curves}. In these examples,
the specialization index explains the absence of a rational point, while the
 index does not.
In Section \ref{sec-cohomology}, we prove our main result, Theorem \ref{thm-zero}. The proof is based on Hodge theory, the Woods Hole trace formula and a description of the behavior
 of $snc$-models under ramified base change.

\subsection*{Acknowledgements} We are grateful to H\'el\`ene Esnault and Olivier Wittenberg for stimulating discussions on the content of this paper, and for valuable comments.
In particular, Olivier Wittenberg made several useful suggestions to improve an earlier version of this paper. The research of the first author was supported by a PhD fellowship of the Research Foundation of Flanders (FWO).

\subsection*{Notation and terminology}
\sss Let $R$ be a henselian discrete valuation ring with
algebraically closed residue field $k$ and  quotient field $K$.
We fix an algebraic closure $K^a$ of $K$. The
 integral closure $R^a$ of $R$ in $K^a$ is a valuation ring with
quotient field $K^a$ and residue field $k$ (see \cite[II.6.2]{neukirch}). We denote by $p$ the characteristic exponent of $k$.
 %and by $\N'$ the set of positive integers prime to $p$.

\sss Let $L$ be a finite extension of $K$. Since $K$ is henselian, the discrete valuation on $K$ extends uniquely to a valuation on $L$, whose valuation ring is the integral closure of $R$ in $L$ \cite[II.6.2]{neukirch}. If we denote by $e$ the ramification index of $L$ over $K$, then Ostrowski's Lemma implies that
 $[L:K]=\delta e$ with $\delta$ a power of $p$ (see Theorem 2 in \cite[6.1]{ribenboim}). The invariant $\delta$ is called the {\em defect} of the extension $L/K$. We say that the extension $L/K$ is defectless if $\delta=1$; this is always the case  if $L$ is separable over $K$ or $R$ is excellent (in particular, if $R$ is complete), as can be deduced from Theorem 2 and Corollary 1 in \cite[VI.8.5]{bourbaki}.

\sss For every $R$-scheme $\cX$, we denote by $\cX_K=\cX\times_R K$ its generic fiber and by $\cX_k=\cX\times_R k$ its special fiber.
 If $\cX$ is proper over $R$, we can
 consider the specialization
map $$\spe_{\cX}:\cX_K(K^a)\to \cX_k(k).$$ It is defined
by composing the inverse of the bijection $\cX(R^a)\to \cX_K
(K^a)$ with the reduction map $\cX(R^a)\to
\cX_k(k)$.

\sss Let $X$ be a proper $K$-scheme. A \textit{model} for $X$ is a proper flat
$R$-scheme $\cX$ endowed with an isomorphism of $K$-schemes
$\cX_K\to X.$  Nagata's compactification theorem implies
that a model of $X$ always exists. A morphism of models is defined in the obvious way. If $\cX$ and $\cX'$ are models of $X$ then we say that $\cX'$ dominates $\cX$ if there exists a morphism of models $\cX'\to \cX$ (such a morphism is unique if it exists). The domination relation defines a partial ordering $\geq$ on the set of isomorphism classes of $R$-models of $X$.

\sss \label{sss:res} We say that $\cX$ is an $snc$-model of $X$ if $\cX$ is regular and $\cX_k$
is a divisor with strict normal crossings. If $p=1$, then every regular proper $K$-scheme $X$ has an $snc$-model by Hironaka's resolution of singularities. More precisely, for
every model $\cX$ of $X$ there exists a morphism of models $\cX'\to \cX$
  such that $\cX'$ is an $snc$-model of $X$\footnote{More precisely, Hironaka showed that we can find such a morphism as a composition of blow-ups at centers in the special fiber if $R$ is complete. The general case then follows from the fact that blowing up commute with flat base change and that the $snc$ property can be checked after base change to the completion.}.
  By \cite[1.1]{cossartpiltant2014} and \cite[0.3]{cossart-jannsen-saito}, this also holds for $p\geq 2$ provided that $R$ is excellent and the dimension of $X$ is at most two. If $p=1$ or $X$ is a curve, one can arrange that $\cX'\to \cX$ is an isomorphism over the $snc$-locus of $\cX$, but this does not seem to follow directly from the statements in  \cite{cossart-jannsen-saito} and  \cite{cossartpiltant2014}.

\sss \label{not-snc} For every regular proper $K$-scheme $X$ and every $snc$-model $\cX$ of $X$, we introduce the following notation.
  We write $$\cX_k=\sum_{i\in I}N_iE_i$$ with $E_i$, $i\in I$ the irreducible components of $\cX_k$ and $N_i$ the multiplicity of $E_i$ in the divisor $\cX_k$. For every non-empty subset $J$ of $I$ we set $$E_J=\bigcap_{j\in J}E_i,\quad E_J^o=E_J\setminus \left(\bigcup_{i\notin J}E_i\right).$$ The sets $E_J$ form a partition of $\cX_k$ into locally closed subsets. We also write
  $$N_J=\mathrm{gcd}\{N_j\,|\,j\in J\}.$$

\section{The specialization index}\label{sec-def}

\begin{definition}[The specialization index]
For every proper $R$-scheme $\cX$, we denote by $\mathcal{D}_{\cX}$
the set of integers $d>0$ such that there exists a zero-cycle of degree
$d$ on $\cX_K$ whose support is contained in $\spe_{\cX}^{-1}(x)$ for
some $x\in \cX(k)$. We call the minimum of $\mathcal{D}_{\cX}$ the
 specialization index of $\cX$, and we denote it by $\iota_{\spe}(\cX)$.

For every proper $K$-scheme $X$, we set
$$\mathcal{D}_{X}=\bigcap \{\mathcal{D}_{\cX}\,|\,\cX\mbox{ is an }R\mbox{-model of }X\}.$$
We call the minimum of $\mathcal{D}_{X}$ the specialization index
of $X$, and we denote it by $\iota_{\spe}(X)$.
\end{definition}

\sss Note that $\mathcal{D}_X$ is non-empty, because for every closed point
 $x$ of $X$, the degree of $x$ belongs to $\mathcal{D}_X$.  It is obvious
 from the definition that $\iota_{\spe}(X)\geq \iota_{\spe}(\cX)$ for every
 model $\cX$ of $X$, and that $\mathcal{D}_{\cX'} \subseteq \mathcal{D}_{\cX}$ (and hence $\iota_{\spe}(\cX')\geq \iota_{\spe}(\cX)$) if $\cX$ and $\cX'$ are models of $X$ such that $\cX'$ dominates $\cX$.
 If we denote by $\iota(X)$ the index of $X$ and by $\nu(X)$ its $\nu$-invariant (that is, the minimal degree of a closed point of $X$), then
  we obviously have
 $\iota(X)|\,\iota_{\spe}(X)$ and $\iota_{\spe}(X)\leq \nu(X)$. In particular, if $X$ has a rational point, then $\iota(X)=\iota_{\spe}(X)=1$.

\sss At first sight, it seems difficult to compute the specialization index $\iota_{\spe}(X)$ of a
 proper $K$-scheme $X$, since the definition suggests that we have to
 check all possible models of $X$. We will show in Proposition
 \ref{prop-indsncd} that, if every model of $X$ can be dominated by
an  $snc$-model, the specialization index can be computed on a single
$snc$-model. First, we need to establish some auxiliary results.

\begin{lemma}\label{lem:pointwithmult}
Let $X$ be a proper $K$-variety and let $\cX$ be a model of $X$. Suppose that $y \in \cX_k$ is a closed point such that both $\cX$
and $(\cX_k)_{\red}$ are regular at $y$, and denote by $N$ the multiplicity of $\cX_k$ at $y$.
  Then there is a closed point $x \in X$ that
specializes to $y$ and such that residue field at $x$ is a defectless extension of $K$ of degree $N$.
 \end{lemma}
 \begin{proof}
   The following proof is taken from \cite[4.6]{wittenberg}. Let $\pi$ be a uniformizer in $R$. Then we can find elements $u$ and $t$ in $\mathcal{O}_{\cX,y}$ such that $u$ is a unit, $t=0$ is a local equation for $(\cX_k)_{\red}$ at $y$, and
  $\pi=ut^N$. We can extend $t$ to a regular system of local parameters
 $(t, f_1, \ldots, f_n)$ in $\O_{\cX,y}$. We set $$R'=\frac{\O_{\cX,y}}{(f_1, \ldots, f_n)}.$$
 Then $R'$ is a regular local ring of dimension 1, and thus
 a discrete valuation ring. The residue class of  $t$ is a uniformizer in $R'$ and its residue field $R'/(t)$ is isomorphic to $k$, the residue field of $\cX$ at $y$.
 The natural morphism $R \to R'$ is flat because it maps $\pi$ to the non-zero element $ut^N$.

 Take $U=\Spec A$ an affine open around $y$ such that $f_1,\ldots, f_n$ and $t$ are defined on $U$. If we take
 $$V = \Spec \frac{A}{(f_1, \ldots, f_n)},$$ then $y \in V$ and $\O_{V,y} = R'$. Now $f:V \to \Spec R$ is of finite
 type and the image of  $y$ is the closed point of $\Spec R$.  Thus, \cite[18.5.11(c')]{EGAIV4}
 implies that  $R'$ is a finite $R$-algebra. Since $R'$ is flat over $R$ and $R'/\pi R'  \cong R'/(t^N)$, we know that it is free
 of  rank $N$.
 Hence, the fraction field $K'$ of $R'$ is  an extension of $K$ of degree $N$.
 If we denote by $x$ the image of the generic point of $\Spec R'$ under the natural morphism
$$\Spec R'\to  \cX,$$ then $x$ is a closed point of degree $N$ on $X$ and $\spe_{\cX}(x)=y$.
 \end{proof}

\begin{prop}\label{prop-sncd}
Let $X$ be a proper  $K$-scheme and let $\cX$ be a model for $X$.
 Let $y$ be a
  closed point of $\cX_k$ such that $\cX$ is regular at $y$ and $\cX_k$ has strict normal crossings at $y$.
  Denote by $E_j$, $j\in J$ the irreducible components of $\cX_k$ that pass through $y$ and by $N_j$ the multiplicity of $\cX_k$ along $E_j$.

  Then for every positive integer $d$,
    the following properties
are equivalent:
\begin{enumerate} \item \label{it:spe} there exists a closed point $x$ on $X$ whose degree divides $d$ and such that
$\spe_{\cX}(x)=y$; \item \label{it:combi} there exists a
tuple $\alpha$ in $(\Z_{>0})^J$ such that $$\sum_{j\in J}\alpha_j N_j=d.$$
\end{enumerate}
\end{prop}
\begin{proof}
 Assume that \eqref{it:spe} holds.
 We choose a uniformizer $\pi$ of $R$. Since
$\cX$ is regular at $y$, we can choose a local equation $z_j=0$ for $E_j$ at $y$, for every $j\in J$, and write
$$\pi=u\prod_{j\in J}z_j^{N_j}$$ in $\mathcal{O}_{\cX,y}$, with $u$ a unit. If we denote by $v$ the normalized discrete valuation on the residue field at $x$, we find
 \begin{equation}\label{eq:val}
 v(\pi)=\sum_{j\in J}v(z_j(x))N_j.\end{equation} But $v(\pi)$ divides the degree of $x$, and therefore $d$. Moreover,
 $v(z_j(x))>0$ for every $j$ in $J$ since $x$ specializes to a point in the zero locus of $z_j$. Thus, \eqref{it:combi} follows by multiplying \eqref{eq:val} with $d/v(\pi)$.

 Now suppose, conversely, that (2) holds. By Lemma \ref{lem:pointwithmult}, it suffices to find a morphism
 $h:\cX'\to \cX$ of models of $X$ and a closed point $y'$ on $\cX_k'$ such that $h(y')=y$, $\cX'$ and $(\cX')_{\red}$ are regular at $y'$, and the multiplicity of $\cX'_k$ at $y'$ divides $d$.
 We will construct $h$ by means of a toroidal blow-up. The most convenient way to express the construction is the language of log geometry; we will in particular make use of the notion of log regularity developed in \cite{kato1994toric}. For every flat $R$-scheme of finite type $\cY$, we denote by $\cY^+$ the log scheme that we obtain by endowing $\cY$ with the divisorial log structure associated with $\cY_k$. Then $\cX^+$ is log regular at $y$ and the characteristic monoid $\overline{\mathcal{M}}_{\cX^+,y}$ is canonically isomorphic to $\N^J$, because we assumed that $\cX_k$ is a strict normal crossings divisor at $y$. The residue class of $\pi$ in $\overline{\mathcal{M}}_{\cX^+,y}$ is the multiplicity vector $m=(N_j)_{j\in J}$.
  We subdivide the cone $\R_{\geq 0}^J$ into a fan $\Sigma$ by adding the ray generated by $\alpha=(\alpha_j)_{j\in J}$. By \cite[10.3]{kato1994toric}, this subdivision gives rise to a proper birational morphism $h:\cX'\to \cX$ that is an isomorphism over the complement of $y$ and such that $(\cX')^+$  is log regular (in particular, normal) at every point of $E=h^{-1}(y)$. Moreover, $E$ is an irreducible divisor on $\cX'$ and the multiplicity of $\cX'_k$ along $E$ is given by the scalar product $m\cdot \alpha'$
     where $\alpha'$ is a primitive generator of the ray $\R_{\geq 0}\alpha$. By our assumption that $m\cdot \alpha=d$, we know that $m\cdot \alpha'$ divides $d$. Thus, we can take for $y'$ any closed point on $E$ such that $\cX'$ and $(\cX'_k)_{\red}$ are regular at $y'$. Such a point exists because $\cX'$ is normal and hence regular in codimension one.
 \end{proof}

\begin{cor}\label{cor-formulas}
Let $X$ be a regular proper $K$-scheme and assume that $X$ admits an $snc$-model $\cX$, with $\cX_k=\sum_{i\in I}N_iE_i$.
 Then
$$\mathcal{D}_{\cX}=\bigcup_{\emptyset \neq J \subset I, E_J^o \neq \emptyset}N_J\cdot
\Z_{>0}.$$
Moreover, we
have the following equalities:
\begin{eqnarray*}
\nu(X)&=&\min \{N_i\,|\,i\in I\},
\\ \iota(X)&=& \gcd \{N_i\,|\,i\in I\},
\\ \iota_{\spe}(\cX)&=& \min \{N_J\,|\,\emptyset \neq J\subset I,\
E_J^o\neq \emptyset\}.
\end{eqnarray*}
\end{cor}
\begin{proof}
 This follows at once from Proposition \ref{prop-sncd} (and the formulas for $\nu(X)$ and $\iota(X)$ are well-known).
\end{proof}

We will now prove that the specialization index can be computed on any
$snc$-model. For this, we will need the following combinatorial lemma.

\begin{lemma}
Let $S$ be a subset of $\Z_{>0}$ and suppose that $S$ is the union of
finitely many sub-semigroups $S_1, \ldots, S_t$ of $(\N,+)$. If we set
$s_i = \gcd S_i$ for every $i\in \{1,\ldots,t\}$,
 then the value  $$s_{\min}=\min\{s_i\mid i\in \{1,\ldots,t\}\,\}$$ is independent of the choice of the $S_i$.
\label{lem-comb}
\end{lemma}
\begin{proof}
 Let $S'$ be a sub-semigroup of $(\N,+)$ that is contained in $S$. We set $s'=\gcd S'$. Then it suffices to show that
 $s_{\min}\leq s'$.  We will prove that, for every sufficiently large positive integer $m$, the element $ms'$ lies in $S'$, and thus in one of the semigroups $S_i$.
  This means that $ms'$ is divisible by one of the $s_i$ for every sufficiently large $m$, so that $s'$ itself must be divisible by one of the $s_i$ and, in particular,
  $s_{\min} \leq s'$.

   By Bezout's lemma, we can write $$s' = \sum_{j=1}^n \lambda_j x_j$$ with
$\lambda_j \in \Z$ and $x_j \in S'$. Reordering terms, we may assume that there exists an integer
 $q$ such that $1\leq q\leq n$, $\lambda_j>0$ for $j\leq q$, and $\lambda_j<0$ for $j>q$. Now set
 $$m_0=-(x_1-1)\sum_{j=q+1}^n\lambda_j(x_j/s')$$ and choose any integer $m> m_0$.
  We consider the Euclidean division $m-m_0=\alpha (x_1/s') +r$ with $\alpha$ a non-negative integer and $r$ an element in $\{0,\dots,(x_1/s')-1\}$.
  Then we can write
\begin{eqnarray*}
ms'&=&\alpha x_1 + rs' - (x_1-1)\sum_{j=q+1}^n \lambda_j x_j
\\ &=& \alpha x_1 +r\sum_{j=1}^q \lambda_jx_j - (x_1-1-r)\sum_{j=q+1}^n \lambda_j x_j
\end{eqnarray*}
 which is a positive linear combination of the elements $x_1,\ldots,x_n$ of the semi-group $S'$. It follows that $ms'$ lies in $S'$ as soon as $m> m_0$.
 \end{proof}

\begin{prop}\label{prop-indsncd} Let $X$ be a regular proper $K$-scheme and assume that every model of $X$ can be dominated by
an  $snc$-model. Then for every $snc$-model $\cX$ of $X$,
 we have $\iota_{\spe}(X)=\iota_{\spe}(\cX)$.
\end{prop}
\begin{proof}
 We have already observed that $\mathcal{D}(\cX')\subseteq \mathcal{D}(\cX)$ if $\cX'$ is a model of $X$ that dominates $\cX$. Thus we only need to
 argue that $\iota_{\spe}(\cX)$ is independent of the choice of the $snc$-model $\cX$.
 We write
 $$\cX_k=\sum_{i\in I}N_iE_i$$ and we
  denote by $S$ the subset of $\Z_{>0}$ consisting of all the multiples of the degrees of the closed points on $X$.
 For
every non-empty subset $J$ of $I$, we set
$$S_J = \left\{ \sum_{j \in J} \alpha_j N_j \mid \alpha_j \in \Z_{>0} \mbox{ for all }j\in J\right\}\subset \Z_{>0}.$$
 Then $N_J=\gcd S_J$. We denote by $\mathcal{J}$ the set of non-empty subsets $J$ of $I$ such that $E_J^o\neq \emptyset$.
 If $J$ belongs to $\mathcal{J}$, then by Proposition \ref{prop-sncd}, the set $S_J$ consists precisely of the multiples of the degrees of the closed points
on $X$ that specialize to a point on $E^o_J$. Since the sets $E^o_J$ form a partition of $\cX_k$, we find that
  $$S = \bigcup_{J\in \mathcal{J}} S_J.$$  Now Lemma
\ref{lem-comb} implies that  $\min \{N_J\mid J\in \mathcal{J}\}$
 is independent of the choice of the $snc$-model $\cX$; but this is precisely $\iota_{\spe}(\cX)$, by Corollary \ref{cor-formulas}.
\end{proof}
\begin{cor}\label{cor-computsp}
Let $X$ be a regular proper $K$-scheme and assume that every model of $X$ can be dominated by
an  $snc$-model. If $\cX$ is an $snc$-model of $X$, with
$$\cX_k=\sum_{i\in I}N_i E_i,$$ then the specialization index of $X$ can be computed as
$$\iota_{\spe}(X)= \min \{N_J\,|\,\emptyset \neq J\subset I,\
E_J^o\neq \emptyset\}.$$
\end{cor}
\begin{proof}
This is a combination of Proposition \ref{prop-indsncd} and Corollary \ref{cor-formulas}.
\end{proof}
 We recall that the assumption in the statements of Proposition \ref{prop-indsncd} and Corollary \ref{cor-computsp} is satisfied when $p=1$, or $\mathrm{dim}(X)\leq 2$ and $R$ is excellent (see \eqref{sss:res}). One can reduce to the excellent case by means of the following result.

\begin{prop}\label{prop:complete}
Denote by $\widehat{R}$ the completion of $R$ and by $\widehat{K}$ its quotient field. Let $X$ be a regular proper $K$-scheme, and assume that every $\widehat{R}$-model of $X\times_K \widehat{K}$ can be dominated by a regular $\widehat{R}$-model.
 Then $\iota_{\spe}(\cX)=\iota_{\spe}(\cX\times_R \widehat{R})$ for every $R$-model $\cX$ of $X$, and
 $\iota_{\spe}(X)=\iota_{\spe}(X\times_K \widehat{K})$.
\end{prop}
\begin{proof}
 We first show that the base changes to $\widehat{R}$ of the $R$-models of $X$ are cofinal in the partially ordered set of isomorphism classes of $\widehat{R}$-models of $X\times_K \widehat{K}$. If we fix any $R$-model $\cX$ of $X$, then we can dominate every $\widehat{R}$-model of $X\times_K \widehat{K}$ by an admissible blow-up of $\cX\times_R \widehat{R}$ (that is, the blow-up of an ideal that contains a power of the maximal ideal of $R$; this is one of the cornerstones of Raynaud's formal-rigid geometry \cite[4.1]{formrig}, and it also holds for proper algebraic schemes by formal GAGA and the fact that admissible blow-ups commute with formal completion). But every admissible blow-up of
 $\cX\times_R \widehat{R}$ is defined over $R$, since blowing up commutes with flat base change. This implies that it is enough to prove the first assertion in the statement.

 Let $\cX$ be an $R$-model of $X$. It is obvious that $\iota_{\spe}(\cX)\geq \iota_{\spe}(\cX\times_R \widehat{R})$. Thus it suffices to show that, for every positive integer $d$ and every closed point $\widehat{x}$ of degree $d$ on $X\times_K \widehat{K}$, there exists a closed point
  $x$ on $X$ whose degree divides $d$ and such that
$$\spe_{\cX\times_R \widehat{R}}(\widehat{x})=\spe_{\cX}(x).$$  Let $\cY$ be a regular $\widehat{R}$-model of $X$ that dominates $\cX\times_R \widehat{R}$.
 By repeatedly blowing up $\cY$ at the specialization of $\widehat{x}$, we can arrange that $\spe_{\cY}(\widehat{x})$ lies on a unique irreducible component of $\cY_k$, whose multiplicity $N$ must divide $d$ (note that the multiplicity of the exceptional divisor will strictly increase as long as the center of the blow-up lies on more than one component; but this multiplicity is also bounded by $d$). Let $\cY'$ be the blow-up of $\cY$ at $\spe_{\cY}(\widehat{x})$, and denote by $E'$ the exceptional divisor.
  By the first part of the proof, we can dominate $\cY'$ by a model $\cY"$ that is defined over $R$ (but which is not necessarily regular);
   we write $\cY"\cong \cZ\times_R \widehat{R}$ for some model $\cZ$ of $X$ over $R$. If we denote by $E"$ the strict transform of $E'$ in $\cY"$, then we can also view $E"$ as an irreducible component of $\cZ_k$. This component still has multiplicity $N$, and $\cZ$ is regular at the generic point of $E"$ since $\cY'$ is regular at the generic point of $E'$. Now it follows from Lemma \ref{lem:pointwithmult} that there exists a point $x$ on $X$ of degree $N$ such that $\spe_{\cZ}(x)$ lies on $E"$, which implies that
 $$\spe_{\cX\times_R \widehat{R}}(\widehat{x})=\spe_{\cX}(x).$$
\end{proof}

\section{Specialization indices of curves}\label{sec-curves}

\sss In this section, we will show that the specialization index $\iota_{\spe}$ is a
strictly finer invariant than the index $\iota$.
 Our examples will all be
constructed with the help of Winters's classification of reduction types of curves. Let $X$ be a smooth, proper, geometrically connected $K$-curve, and let $\cX$ be an $snc$-model of $X$. We write $$\cX_k=\sum_{i=1}^r N_iE_i,$$ we denote by $G_i$ the genus of $E_i$, for every $i$, and we set $N=(N_1,\ldots,N_r)$ and $G=(G_1,\ldots,G_r)$. Moreover, we set $c_{ij}=E_i\cdot E_j$ for all $i,j$. Then the matrix
$C=(c_{ij})_{i,j}$ is a symmetric $r\times r$ matrix with integer coefficients, and $C\cdot N^t=0$ because $E_i\cdot \cX_k=0$ for every $i$. We say that $\cX$ has {\em reduction type} $(N,G,C)$.

\begin{theorem}[Winters]\label{thm:winters}
Assume that $k$ has characteristic zero and that $R$ is the henselization of $\A^1_k$ at the origin. Let $r$ be a positive integer, let $N$ be a primitive vector in $\Z_{>0}^r$, let $G$ be an element of $\N^r$ and let $C=(c_{ij})$ be a symmetric $r\times r$ matrix with integer coefficients  such that $c_{ij}\geq 0$ for $i\neq j$ and $C\cdot N^t=0$. Then there exist a
 smooth, proper, geometrically connected $K$-curve $X$ and an $snc$-model $\cX$ of $X$ of reduction type $(N,G,C)$.
  Moreover, the genus $g(X)$ of $X$ can be computed from the formula $$2-2g(X)=\sum_{i=1}^r N_i(2-2G_i+c_{ii}).$$
\end{theorem}
\begin{proof}
The existence of $X$ and $\cX$ follows from \cite[3.7 and 4.3]{winters}. The formula for the genus $X$ is an elementary consequence of Riemann-Roch and adjunction; see for instance \cite[3.1]{nicaisetame}.
\end{proof}
 Note that Theorem \ref{thm:winters} remains valid if $k$ has characteristic zero and $R$ is complete, since we can choose an isomorphism between $R$ and the completion of the local ring of $\A^1_k$ at the origin and perform a base change on the model given by Theorem \ref{thm:winters}.

\begin{example}\label{ex:examplestrict}
Assume that $k$ has characteristic zero and that $K$ is complete.
 We choose a non-negative integer $x$. We set $r = 5$ and

{\small
$$ \begin{pmatrix} N_1 \\ N_2 \\ N_3 \\ N_4 \\ N_5 \end{pmatrix} = \begin{pmatrix} 2 \\  2 \\ 3 \\ 4 \\ 6 \end{pmatrix},
\quad
\begin{pmatrix} G_1 \\ G_2 \\ G_3 \\ G_4 \\ G_5 \end{pmatrix} = \begin{pmatrix} x \\  0 \\ 0 \\ 0 \\ 0 \end{pmatrix},
\quad
C = \begin{pmatrix} -3 & 0 & 0 & 0 & 1 \\   0 & -2 & 0 &1 & 0 \\ 0 & 0 & -4 & 0 & 2 \\ 0 & 1 & 0 &-2 & 1 \\ 1 & 0 & 2 & 1 & -2 \end{pmatrix}.$$}

 Let $\cX$ be an $snc$-model of reduction type $(N,G,C)$, with generic fiber $X$. Then we can compute by means of Corollaries \ref{cor-formulas} and \ref{cor-computsp} that $\iota(X)=1$ whereas $\iota_{\spe}(\cX)=2$.
 Using the formula in Theorem \ref{thm:winters}, we also see that $g(X) = 5+2x$.

 A picture of the dual graph of the special fiber can be found in Figure  \ref{fig:examplestrict}.

\begin{figure}[!ht]
\centering
\begin{tikzpicture}
[scale=.28,auto=left,minimum size=1cm,every node/.style={ellipse,draw}]
\node (n2) at (1,9) {$(2,0)$};
\node (n1) at (12,9) {$(2,x)$};
\node (n3) at (24,9)  {$(3,0)$};
\node (n4) at (1,2)  {$(4,0)$};
\node (n5) at (12,2)  {$(6,0)$};
\draw (n4) -- (n5);
\draw (n1) -- (n5);
\draw (n2) -- (n4);
\path (n3) edge [bend left = 22] (n5);
\path (n3) edge [bend right =  22] (n5);
\end{tikzpicture}
\caption{The dual graph of the special fiber $\cX_k$ of Example \ref{ex:examplestrict}. The vertex corresponding to a component $E_i$ is labelled with $(N_i,G_i)$.} \label{fig:examplestrict}
\end{figure}
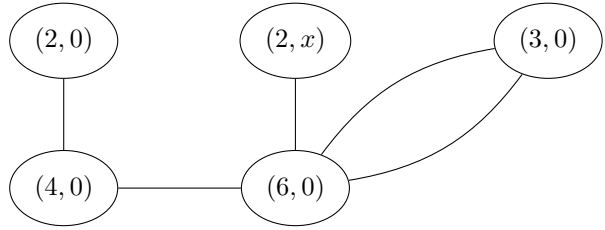

\end{example}

\begin{example}\label{ex:examplestrict2}
Assume that $k$ has characteristic zero and that $K$ is complete.
 We choose a non-negative integer $x$. We set $r = 6$ and

{\small $$ \begin{pmatrix} N_1 \\ N_2 \\ N_3 \\ N_4 \\ N_5 \\ N_6  \end{pmatrix} = \begin{pmatrix} 2 \\  2 \\ 3 \\  3 \\ 4 \\ 6 \\ \end{pmatrix},
  \quad
  \begin{pmatrix} G_1 \\ G_2 \\ G_3 \\ G_4 \\ G_5 \\ G_6 \end{pmatrix} = \begin{pmatrix} x \\  0 \\ 0 \\ 0 \\ 0 \\ 0 \end{pmatrix},
  \quad
  C = \begin{pmatrix} -3& 0 & 0 & 0 & 0 & 1 \\   0 & -2 & 0 & 0 &1 & 0 \\ 0 & 0 & -2 & 0 & 0 & 1 \\ 0 & 0 & 0 & -2&0 & 1 \\ 0 & 1 & 0 & 0 & -2 &1  \\ 1 & 0 & 1 & 1 & 1 & -2  \end{pmatrix}.$$}

 Let $\cX$ be an $snc$-model of reduction type $(N,G,C)$, with generic fiber $X$. Then we again find that $\iota(X)=1$ and $\iota_{\spe}(\cX)=2$.
 In this example, $g(X) = 2+2x$.

A picture of dual graph of the special fiber can be found in Figure  \ref{fig:examplestrict2}.

\begin{figure}[!ht]
\centering
\begin{tikzpicture}
[scale=.28,auto=left,every node/.style={ellipse,draw}]
\node (n1) at (1,9) {$(2,x)$};
\node (n4) at (12,9) {$(3,0)$};
\node (n3) at (24,9)  {$(3,0)$};
\node (n5) at (1,2)  {$(4,0)$};
\node (n6) at (12,2)  {$(6,0)$};
\node (n2) at (-12,2)  {$(2,0)$};
\draw (n1) -- (n6);
\draw (n5) -- (n6);
\draw (n3) -- (n6);
\draw (n4) -- (n6);
\draw (n2) -- (n5);
\end{tikzpicture}
\caption{The dual graph of the special fiber $\cX_k$ of Example \ref{ex:examplestrict2}. The vertex corresponding to a component $E_i$ is labelled with $(N_i,G_i)$.} \label{fig:examplestrict2}
\end{figure}
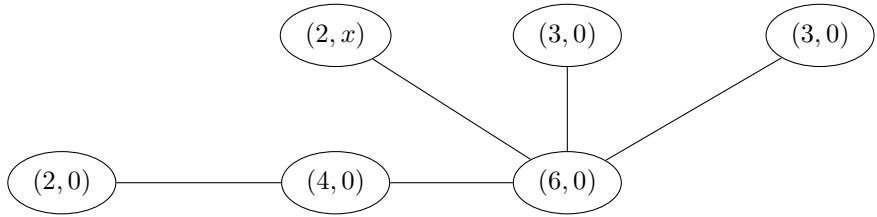
\end{example}

Combining Examples \ref{ex:examplestrict} and \ref{ex:examplestrict2}, we arrive at the following proposition.

\begin{prop}\label{prop:strict}
Assume that $k$ has characteristic zero and that $K$ is complete. If $g=2$ or $g \geq 4$, then
 there exists a smooth, proper, geometrically
connected $K$-curve $C$ of genus $g$ such that $\iota(C)=1$ while
$\iota_{\spe}(C)>1$.
\end{prop}
\begin{proof}
This follows from Examples \ref{ex:examplestrict} and \ref{ex:examplestrict2}, since we can write $g$ as $2+2x$ or $5+2x$ with $x$ a non-negative integer.
\end{proof}

\sss \label{sss:lowgenus} Proposition \ref{prop:strict} does not cover the cases $g\leq 1$ and $g=3$. As a matter of fact, we will now explain that in those cases, the specialization index  is one if and only if the index is one. Let $X$ be a smooth, proper and geometrically connected $K$-curve. If $X$ has genus zero then it has a rational point by the triviality of the Brauer group of $K$, so that $\iota(X)=\iota_{\spe}(X)=1$. If $X$ has genus one and $D$ is a divisor of degree $\iota(X)$ on $X$, then it follows from  Riemann-Roch that $D$ is linearly equivalent to an effective divisor $D'$ on $X$. Then any closed point in the support of $D'$ has degree at most $\iota(X)$, so that
 $\iota(X)=\iota_{\spe}(X)=\nu(X)$. This value is also equal to the period of $X$ (the order of the class of $X$ in the Weil-Ch\^atelet group of its Jacobian) by the triviality of the Brauer group of $K$ \cite[Theorem 1]{lichtenbaum}.
 The genus 3 case is treated in the following proposition.

\begin{prop} Let $X$ be a smooth, proper, geometrically connected $K$-curve of genus $3$.
If $\iota(X) = 1$, then $\iota_{\spe}(X)=1$.
\end{prop}
\begin{proof}
 We have seen in Corollaries \ref{cor-formulas} and \ref{cor-computsp} that $\iota(X)$ and $\iota_{\spe}(X)$ only depend on the dual graph of the special fiber of an $snc$-model $\cX$ of $X$, where we label each vertex with the multiplicity of the corresponding component. The assumption that $\iota(X)=1$ is equivalent to the property that the multiplicity vector $N$ of $\cX$ is primitive, by Corollary \ref{cor-formulas}.
  Denote by $K'$ the fraction field of the henselization of $\A^1_k$ at the origin. By Winters's theorem (Theorem \ref{thm:winters}), we can always find a smooth, proper and geometrically connected curve $X'$ of genus $3$ over $K'$ with an $snc$-model whose special fiber has the same dual graph (with the same labels) as $\cX_k$. Thus we may assume that $K=K'$.  Then
 we can deduce the statement from the classification of reduction types of minimal regular models of genus $3$ curves over complex discs in
  \cite{uematsu1999} (alternatively, one can also check that the results we need from that paper remain valid for arbitrary $R$, with the same arguments).

  We must show that there is no case where $\iota(X)=1$ but $\iota_{\spe}(X)>1$. This means that we only need to look
   at minimal regular models $\cX$ whose special fiber does not contain an irreducible component of multiplicity one, or two irreducible components of coprime multiplicities
    that intersect transversally in a point (in these cases, $\iota_{\spe}(X)=1$ by Corollary \ref{cor-computsp} and the fact that we can modify $\cX$ into an $snc$-model without changing its $snc$-locus).

    The cases left to consider are the following combinations of trunks: [2G], [2H], [4I], [4J], [2I,2I], [2I,2J], [2J,2J] (see \cite[Figure 3]{uematsu1999}). In order to have $\iota(X)=1$, the special fiber of the minimal regular model $\cX$ must contain a branch of type $(b)$ with $p=6$ or $(h)$ with $p=6$ and $q=4$ (see \cite[Figure 2]{uematsu1999}). All the components of the other branches have multiplicities divisible by 2. This is only possible with the trunk combinations [2H] and [2J,2J].   In both cases, we would then need an additional branch of type $(b)$ with $p=2$, but these branches have a component with multiplicity one, so that $\iota_{\spe}(X)=1$.
\end{proof}

%\sss Note, however, that there exist genus $3$ curves where the specialization index is strictly larger than the index (and both are different from one). An example will be given %in Example \ref{ex:sharpbound}.

\sss To conclude this section, we will establish an upper bound for the specialization index of a curve in terms of its genus.
 Let $X$ be a smooth, proper, geometrically connected $K$-curve of genus $g$. We have already observed in \eqref{sss:lowgenus} that $X$ has a rational
point (thus $\iota(X)=\iota_{\spe}(X)=1$) if $g=0$, and that $\iota(X)$ and $\iota_{\spe}(X)$ are both equal to the period of $X$ if $g=1$. This period can be arbitrarily large: if $E$ is an elliptic $K$-curve with semi-stable reduction, then $H^1(K,E)$ contains elements of arbitrarily large order, by \cite{shafarevich}. However, if $g\geq 2$, we obtain an
upper bound on $\iota(X)$ from
 Riemann-Roch: the degree
of the canonical divisor of $X$ equals $2g-2$, so that $\iota(X)\leq 2g-2$. With some extra work, one can bring this bound down to $g-1$; this result is well-known and can be viewed as a special case of \cite[2.1]{esnlevwit} since $1-g$ is the coherent Euler characteristic of $C$. We observe that the bound $2g-2$ still holds for the
specialization index, as a consequence of the following stronger result.

\begin{prop}\label{prop:bound}
Let $F$ be any field and let $X$ be a smooth, proper, geometrically connected $F$-curve of
genus $g\geq 2$. Let $m$ be the smallest positive multiple of the index $\iota(X)$ such that $m\geq g$.
Then $m\leq 2g-2$ and $X$ has a closed point of degree $d\leq m$.
\end{prop}
\begin{proof}
 Every canonical divisor on $X$ has degree $2g-2\geq g$ by Riemann-Roch, so that $m\leq 2g-2$.
 Applying Riemann-Roch to any divisor $D$ of degree $m$ on $X$, we find that $D$ is linearly equivalent to an effective divisor $D'$, and each point in the
 support of $D'$ has degree at most $m$.
 \end{proof}

\sss We do not know if the bound $2g-2$ on the specialization index is sharp for each fixed genus $g\geq 2$. Unfortunately,
 we cannot use Winters's result (Theorem \ref{thm:winters}) to produce examples: the theorem assumes that the multiplicity vector $N$ is primitive, in which case the index $\iota(X)$ equals one and the specialization index is at most $g$, by Proposition \ref{prop:bound}.

\if false
 The following example shows that the bound $2g-2$ for the specialization index is sharp for every $g\geq 2$.

\begin{example}\label{ex:sharpbound}
 Assume that $k$ has characteristic zero and that $K$ is complete. Choose $g\geq 2$, and set

{\small
\begin{equation*}
 \begin{pmatrix} N_1 \\ N_2 \\ N_3 \\ N_4 \\ N_5 \\ N_6  \end{pmatrix} = \begin{pmatrix} 2g-2 \\  2g-2 \\ 3g-3 \\  3g-3 \\ 4g-4 \\ 6g-6 \\ \end{pmatrix},
\quad
 \begin{pmatrix} G_1 \\ G_2 \\ G_3 \\ G_4 \\ G_5 \\ G_6 \end{pmatrix} = \begin{pmatrix} 0 \\  0 \\ 0 \\ 0 \\ 0 \\ 0 \end{pmatrix},
 \quad
 C = \begin{pmatrix} -3& 0 & 0 & 0 & 0 & 1 \\   0 & -2& 0 & 0 &1 & 0 \\ 0 & 0 &-2 & 0 & 0 & 1 \\ 0 & 0 & 0 & -2&0 & 1 \\ 0 & 1 & 0 & 0 &-2  &1  \\ 1 & 0 & 1 & 1 & 1 &-2  \end{pmatrix}.
 \end{equation*} }

By Theorem \ref{thm:winters}, there exists a proper, smooth, geometrically connected curve $X$ over $K$ with an $snc$-model
$\cX$ of type $(N,G,C)$. We easily compute that $\iota(X) = g-1$, $\iota_\spe(X) = 2g-2$ and
$g(X) = g$. Notice that $\iota(X) < \iota_\spe(X)$.

A picture of the dual graph of the special fiber can be found in Figure  \ref{fig:sharpbound}.

\begin{figure}[!ht]
\centering
\begin{tikzpicture}
[scale=.28,auto=left,every node/.style={draw,ellipse}]
\node (n1) at (1,9) {$(2g-2,0)$};
\node (n4) at (12,9) {$(3g-3,0)$};
\node (n3) at (24,9)  {$(3g-3,0)$};
\node (n5) at (1,2)  {$(4g-4,0)$};
\node (n6) at (12,2)  {$(6g-6,0)$};
\node (n2) at (-12,2)  {$(2g-2,0)$};
\draw (n1) -- (n6);
\draw (n5) -- (n6);
\draw (n3) -- (n6);
\draw (n4) -- (n6);
\draw (n2) -- (n5);
\end{tikzpicture}
\caption{The dual graph of the special fiber $\cX_k$ of Example \ref{ex:sharpbound}. The vertex corresponding to a component $E_i$ is labelled with $(N_i,G_i)$.} \label{fig:sharpbound}
\end{figure}
\end{example}
\fi

\section{The specialization index for varieties with trivial coherent cohomology}\label{sec-cohomology}
\sss Throughout this section, we assume that $k$ has characteristic zero (that is, $p=1$) and we denote by $X$ a geometrically connected, proper and smooth $K$-variety.
 Our aim is to prove that  $\iota_\spe(X) = 1$ if $X$ has trivial coherent cohomology. First, we introduce some useful notation. For every positive integer $d$,
 we denote by $K(d)$ the unique extension of degree $d$ of $K$ inside the fixed algebraic closure $K^a$, and we write $R(d)$ for its valuation ring. We also set
 $X(d) = X \times_K K(d)$. The extension $K(d)$ is Galois over $K$, with Galois group $\mu_d(k)$. We choose a geometric monodromy operator, that is, a topological generator $\sigma$ of the procyclic Galois group
 $\mathrm{Gal}(K^a/K)$.

\begin{definition}[Monodromic model]
Let $d$ be a positive integer. A {\em monodromic model} for $X$ over $R(d)$
is a regular model $\cY$ for $X(d)$ over $R(d)$ such that the Galois
action of $\mu_d(k)$ on $X(d)$ extends to
 $\cY$.
 We say that a monodromic model $\cY$ is {\em strictly semi-stable}
if $\cY_k$ is a reduced divisor with strict normal crossings on $\cY$.
\end{definition}

\begin{prop}\label{prop-monmodel}
For every sufficiently divisible positive integer $d$, there exists a strictly semi-stable
monodromic model $\cY$ of $X$ over $R(d)$ such that $\cY$ admits a
$\mu_d(k)$-equivariant morphism $$\cY \rightarrow \cX \times_R R(d)$$ of
 models of $X(d)$.
\end{prop}
\begin{proof}
This can be proven by means of a refinement of the algorithm to construct a semi-stable model in
\cite{KKMS}. A detailed proof for the case where the base is a smooth
$k$-curve is contained in \cite{wang}. Since the proof is essentially combinatorial, it generalizes easily to the case
where the base is $\Spec R$.
\end{proof}

\begin{prop}\label{prop-powerp}
 Let $\cX$ be an $snc$-model of $X$, with $\cX_k=\sum_{i\in I}N_i E_i$, and let $d$ be a positive integer that is divisible by
$$e(\cX)=\lcm\{N_i\mid i\in I\}.$$ We denote by $\cX(d)$ the
normalization of $\cX \times_R R(d)$. If the $\mu_d(k)$-action on $\cX(d)_k(k)$
admits a fixed point, then $\iota_\spe(\cX)=1$. \end{prop}
\begin{proof}
 Let $y$ be a closed point on  $\cX(d)_k$ that is fixed under the action of $\mu_d(k)$, and let $J$ be the unique subset of $I$ such that the image $x$ of $y$ in $\cX_k$ lies in $E_J^o$. A standard computation shows that the fiber $\cX(d)\times_{\cX}x$ of $\cX(d)$ over $x$ consists of $N_J$ points that are transitively permutated by the action of $\mu_d(k)$; see for instance \cite[2.3.2]{nicaisetame}. Since $y$ is fixed under the $\mu_d(k)$-action, we conclude that $N_J=1$, so that $\iota_{\spe}(X)=1$ by Corollary \ref{cor-computsp}.
\end{proof}

\begin{prop}\label{prop-nofix}
If $\iota_\spe(X)>1$, then there exist a positive
integer $d>0$ and a strictly semi-stable monodromic model $\cY$ of $X$ over
$R(d)$ such that the $\mu_d(k)$-action on $\cY_k$ has no $k$-rational fixed point.
 For any such $\cY$, we have
$$\sum_{i\geq 0}(-1)^i\Tr\left(\sigma\mid H^i \left(\cY_k,\mathcal{O}_{\cY_k}\right)\right)=0. $$
\end{prop}
\begin{proof}
The existence of $\cY$ follows from Propositions \ref{prop-monmodel} and
\ref{prop-powerp}. Thus it suffices to prove the vanishing of the trace of $\sigma$ on the coherent cohomology of $\cY_k$.
 We write $$\cY_k=\sum_{i\in I}N_iE_i$$ and we set $E_J=\cap_{j\in J}E_j$ for every non-empty subset $J$ of $I$. We endow the closed subsets $E_J$ of $\cY_k$
 with their reduced induced structure.

 We denote by $n$ the dimension of $X$. For every integer $m$ in $\{1,\ldots,n+1\}$ we set
$$\cY^{(m)}_k=\coprod_{J\subset I,\,|J|=m}E_{J}.$$
We denote by $a_m$ the obvious morphism from $\cY^{(m)}_k$ to
$\cY_k$. The schemes $\cY^{(m)}_k$ are smooth
and proper $k$-varieties, and the $\mu_d(k)$-action on $\cY_k$ induces a
$\mu_d(k)$-action on $\cY^{(m)}_k$ without fixed $k$-points.

The sequence of $\mathcal{O}_{\cY_k}$-modules
$$\mathcal{O}_{\cY_k}\rightarrow  (a_1)_*\mathcal{O}_{\cY^{(1)}_k}\rightarrow
(a_2)_*\mathcal{O}_{\cY^{(2)}_k}\rightarrow \ldots \rightarrow (a_{n+1})_*\mathcal{O}_{\cY^{(n+1)}_k}$$ is exact, and defines a resolution of
 the structure sheaf $\mathcal{O}_{\cY_k}$ that we can use to compute the trace of $\sigma$ on the coherent cohomology of $\cY_k$. By the first spectral sequence for hypercohomology, it suffices to show that
$$\sum_{i\geq 0}(-1)^i \Tr \left(\sigma\mid H^i \left(\cY^{(m)}_k, \mathcal{O}_{\cY^{(m)}_k}\right)\right)=0$$ for every $m$ in $\{1,\ldots,n+1\}$. This follows from the
 Woods Hole trace formula \cite[III.6.12]{sga5}, since $\cY^{(m)}_k$ is smooth and proper over $k$ and $\sigma$ is a finite order automorphism of $\cY^{(m)}_k$ without fixed points.
\end{proof}

\begin{theorem}\label{thm-zero}
Assume that
$$H^i(X,\mathcal{O}_{X})=0$$ for every $i>0$. Then $\iota_\spe(X)=1$. In particular, $\iota(X)=1$, that is, $X$ admits a zero-cycle of
degree one.
\end{theorem}
\begin{proof}
Let $d$ be a positive integer and let $\cY$ be a monodromic semi-stable model for $X$ over $R(d)$.
 We denote by
$$f:\cY\rightarrow \Spec R(d)$$ the structural morphism.
 It follows from Hodge theory that the $R(d)$-modules
$R^if_*\mathcal{O}_\cY$ are free; see for instance  \cite[Theorem 7.1]{IKN} for a statement that applies in our set-up (log-smooth schemes over $R(d)$).
 But we also know that $H^i(X(d), \mathcal{O}_{X(d)}) = 0$ for all $i>0$,
 so  that $$
H^i(\cY_k,\mathcal{O}_{\cY_k})=0$$ for every $i>0$, by  \cite[II.4 Corollary 3]{mumford1974}.
 Moreover, $\sigma$ acts trivially on
$$H^0(\cY_k,\mathcal{O}_{\cY_k}) =k.$$ Thus, we find that
$$\sum_{i\geq 0}(-1)^i \Tr(\sigma\mid H^i (\cY_k,
\mathcal{O}_{\cY_k}))=1.$$
 Proposition \ref{prop-nofix} now implies that $\iota_{\spe}(X)=1$.
\end{proof}

\sss To conclude, we recall that the triviality of the coherent cohomology of $X$ is not a sufficient condition for the existence of a rational point.
 In \cite{Lafon}, Lafon has constructed an example of an Enriques surface over $\C((t))$ without rational points. It remains an important challenge to understand which additional conditions can be imposed to guarantee the existence of a rational point. A natural candidate seems to be the triviality of the geometric fundamental group (which is satisfied for rationally connected varieties but not for Enriques surfaces). Another interesting problem is whether Theorem \ref{thm-zero} remains valid in positive residue characteristic.

\end{document}